\documentclass[a4paper,12pt,draft]{amsart}
\usepackage{amsmath, amssymb, amsthm, url}
\usepackage{braket}

\newcommand{\ops}[2]{\displaystyle\mathop{\vee}_{#1,#2}}
\newcommand{\ZZZ}{\mathfrak{Z}}

\newcommand{\PPP}{\mathcal{P}}

\newcommand{\lvrace}{{\boldsymbol\{}}
\newcommand{\rvrace}{{\boldsymbol\}}}

\newcommand{\ed}[2]{\lvrace #1,#2 \rvrace}

\theoremstyle{plain}   

\newtheorem{thm}{Theorem}[section]
\newtheorem{theorem}[thm]{Theorem}

\newtheorem{lemma}[thm]{Lemma}

\newtheorem{proposition}[thm]{Propsition}

\theoremstyle{remark}  

\newtheorem{remark}[thm]{Remark}

\newtheorem{example}[thm]{Example}

\theoremstyle{definition}  
\newtheorem{definition}[thm]{Definition}

\allowdisplaybreaks[3]

\newcommand{\figI}{
\begin{picture}(160,50)
\put(12,25){\makebox(0,0)[c]{$\cdot$}}
\put(15,30){\makebox(0,0)[bc]{$\overbrace{\quad\quad}^{a_1-1}$}}
\put(15,25){\makebox(0,0)[c]{$\cdot$}}
\put(18,25){\makebox(0,0)[c]{$\cdot$}}
\put(5,25){\makebox(0,0){$\bullet$}}
\put(25,25){\makebox(0,0){$\bullet$}}
\put(15,5){\makebox(0,0){$\bullet$}}
\put(15,5){\line(1,2){10}}
\put(15,5){\line(-1,2){10}}
\put(42,25){\makebox(0,0)[c]{$\cdot$}}
\put(45,30){\makebox(0,0)[bc]{$\overbrace{\quad\quad}^{a_2-1}$}}
\put(45,25){\makebox(0,0)[c]{$\cdot$}}
\put(48,25){\makebox(0,0)[c]{$\cdot$}}
\put(35,25){\makebox(0,0){$\bullet$}}
\put(55,25){\makebox(0,0){$\bullet$}}
\put(45,5){\makebox(0,0){$\bullet$}}
\put(45,5){\line(1,2){10}}
\put(45,5){\line(-1,2){10}}
\put(72,25){\makebox(0,0)[c]{$\cdot$}}
\put(75,30){\makebox(0,0)[bc]{$\overbrace{\quad\quad}^{a_3-1}$}}
\put(75,25){\makebox(0,0)[c]{$\cdot$}}
\put(78,25){\makebox(0,0)[c]{$\cdot$}}
\put(65,25){\makebox(0,0){$\bullet$}}
\put(85,25){\makebox(0,0){$\bullet$}}
\put(75,5){\makebox(0,0){$\bullet$}}
\put(75,5){\line(1,2){10}}
\put(75,5){\line(-1,2){10}}
\put(102,25){\makebox(0,0)[c]{$\cdot$}}
\put(105,30){\makebox(0,0)[bc]{$\overbrace{\quad\quad}^{a_4-1}$}}
\put(105,25){\makebox(0,0)[c]{$\cdot$}}
\put(108,25){\makebox(0,0)[c]{$\cdot$}}
\put(95,25){\makebox(0,0){$\bullet$}}
\put(115,25){\makebox(0,0){$\bullet$}}
\put(105,5){\makebox(0,0){$\bullet$}}
\put(105,5){\line(1,2){10}}
\put(105,5){\line(-1,2){10}}
\put(115,5){\line(-1,0){100}}
\put(125,15){\makebox(0,0)[c]{$\cdots$}}
\put(135,5){\line(1,0){10}}
\put(142,25){\makebox(0,0)[c]{$\cdot$}}
\put(145,30){\makebox(0,0)[bc]{$\overbrace{\quad\quad}^{a_d-1}$}}
\put(145,25){\makebox(0,0)[c]{$\cdot$}}
\put(148,25){\makebox(0,0)[c]{$\cdot$}}
\put(135,25){\makebox(0,0){$\bullet$}}
\put(155,25){\makebox(0,0){$\bullet$}}
\put(145,5){\makebox(0,0){$\bullet$}}
\put(145,5){\line(1,2){10}}
\put(145,5){\line(-1,2){10}}
\end{picture}
}
\newcommand{\figII}{\begin{picture}(160,60)
\put(12,25){\makebox(0,0)[c]{$\cdot$}}
\put(15,30){\makebox(0,0)[bc]{$\overbrace{\quad\quad}^{a_1-1}$}}
\put(15,25){\makebox(0,0)[c]{$\cdot$}}
\put(18,25){\makebox(0,0)[c]{$\cdot$}}
\put(5,25){\makebox(0,0){$\bullet$}}
\put(25,25){\makebox(0,0){$\bullet$}}
\put(15,5){\makebox(0,0){$\bullet$}}
\put(15,5){\line(1,2){10}}
\put(15,5){\line(-1,2){10}}
\put(15,5){\line(3,1){30}}
\put(75,5){\line(-3,1){30}}
\put(42,35){\makebox(0,0)[c]{$\cdot$}}
\put(45,40){\makebox(0,0)[bc]{$\overbrace{\quad\quad}^{a_2-1}$}}
\put(45,35){\makebox(0,0)[c]{$\cdot$}}
\put(48,35){\makebox(0,0)[c]{$\cdot$}}
\put(35,35){\makebox(0,0){$\bullet$}}
\put(55,35){\makebox(0,0){$\bullet$}}
\put(45,15){\makebox(0,0){$\bullet$}}
\put(45,15){\line(1,2){10}}
\put(45,15){\line(-1,2){10}}
\put(72,25){\makebox(0,0)[c]{$\cdot$}}
\put(75,30){\makebox(0,0)[bc]{$\overbrace{\quad\quad}^{a_3-1}$}}
\put(75,25){\makebox(0,0)[c]{$\cdot$}}
\put(78,25){\makebox(0,0)[c]{$\cdot$}}
\put(65,25){\makebox(0,0){$\bullet$}}
\put(85,25){\makebox(0,0){$\bullet$}}
\put(75,5){\makebox(0,0){$\bullet$}}
\put(75,5){\line(1,2){10}}
\put(75,5){\line(-1,2){10}}
\put(102,25){\makebox(0,0)[c]{$\cdot$}}
\put(105,30){\makebox(0,0)[bc]{$\overbrace{\quad\quad}^{a_4-1}$}}
\put(105,25){\makebox(0,0)[c]{$\cdot$}}
\put(108,25){\makebox(0,0)[c]{$\cdot$}}
\put(95,25){\makebox(0,0){$\bullet$}}
\put(115,25){\makebox(0,0){$\bullet$}}
\put(105,5){\makebox(0,0){$\bullet$}}
\put(105,5){\line(1,2){10}}
\put(105,5){\line(-1,2){10}}
\put(115,5){\line(-1,0){100}}
\put(125,15){\makebox(0,0)[c]{$\cdots$}}
\put(135,5){\line(1,0){10}}
\put(142,25){\makebox(0,0)[c]{$\cdot$}}
\put(145,30){\makebox(0,0)[bc]{$\overbrace{\quad\quad}^{a_d-1}$}}
\put(145,25){\makebox(0,0)[c]{$\cdot$}}
\put(148,25){\makebox(0,0)[c]{$\cdot$}}
\put(135,25){\makebox(0,0){$\bullet$}}
\put(155,25){\makebox(0,0){$\bullet$}}
\put(145,5){\makebox(0,0){$\bullet$}}
\put(145,5){\line(1,2){10}}
\put(145,5){\line(-1,2){10}}
\end{picture}
}
\newcommand{\figIII}{\begin{picture}(120,60)
\put(12,25){\makebox(0,0)[c]{$\cdot$}}
\put(15,30){\makebox(0,0)[bc]{$\overbrace{\quad\quad}^{a_1-1}$}}
\put(15,25){\makebox(0,0)[c]{$\cdot$}}
\put(18,25){\makebox(0,0)[c]{$\cdot$}}
\put(5,25){\makebox(0,0){$\bullet$}}
\put(25,25){\makebox(0,0){$\bullet$}}
\put(15,5){\makebox(0,0){$\bullet$}}
\put(15,5){\line(1,2){10}}
\put(15,5){\line(-1,2){10}}
\put(15,5){\line(3,1){30}}
\put(42,35){\makebox(0,0)[c]{$\cdot$}}
\put(45,40){\makebox(0,0)[bc]{$\overbrace{\quad\quad}^{a_2-1}$}}
\put(45,35){\makebox(0,0)[c]{$\cdot$}}
\put(48,35){\makebox(0,0)[c]{$\cdot$}}
\put(35,35){\makebox(0,0){$\bullet$}}
\put(55,35){\makebox(0,0){$\bullet$}}
\put(45,15){\makebox(0,0){$\bullet$}}
\put(45,15){\line(1,2){10}}
\put(45,15){\line(-1,2){10}}
\put(45,15){\line(1,0){30}}
\put(72,35){\makebox(0,0)[c]{$\cdot$}}
\put(75,40){\makebox(0,0)[bc]{$\overbrace{\quad\quad}^{a_3-1}$}}
\put(75,35){\makebox(0,0)[c]{$\cdot$}}
\put(78,35){\makebox(0,0)[c]{$\cdot$}}
\put(65,35){\makebox(0,0){$\bullet$}}
\put(85,35){\makebox(0,0){$\bullet$}}
\put(75,15){\makebox(0,0){$\bullet$}}
\put(75,15){\line(1,2){10}}
\put(75,15){\line(-1,2){10}}
\put(105,5){\line(-3,1){30}}
\put(105,5){\line(-1,0){90}}
\put(102,25){\makebox(0,0)[c]{$\cdot$}}
\put(105,30){\makebox(0,0)[bc]{$\overbrace{\quad\quad}^{a_4-1}$}}
\put(105,25){\makebox(0,0)[c]{$\cdot$}}
\put(108,25){\makebox(0,0)[c]{$\cdot$}}
\put(95,25){\makebox(0,0){$\bullet$}}
\put(115,25){\makebox(0,0){$\bullet$}}
\put(105,5){\makebox(0,0){$\bullet$}}
\put(105,5){\line(1,2){10}}
\put(105,5){\line(-1,2){10}}
\end{picture}}

\newcommand{\figIV}{\begin{picture}(120,60)
\put(12,25){\makebox(0,0)[c]{$\cdot$}}
\put(15,30){\makebox(0,0)[bc]{$\overbrace{\quad\quad}^{a_1-1}$}}
\put(15,25){\makebox(0,0)[c]{$\cdot$}}
\put(18,25){\makebox(0,0)[c]{$\cdot$}}
\put(5,25){\makebox(0,0){$\bullet$}}
\put(25,25){\makebox(0,0){$\bullet$}}
\put(15,5){\makebox(0,0){$\bullet$}}
\put(15,5){\line(1,2){10}}
\put(15,5){\line(-1,2){10}}
\put(15,5){\line(3,1){30}}
\put(15,5){\line(6,1){60}}
\put(42,35){\makebox(0,0)[c]{$\cdot$}}
\put(45,40){\makebox(0,0)[bc]{$\overbrace{\quad\quad}^{a_2-1}$}}
\put(45,35){\makebox(0,0)[c]{$\cdot$}}
\put(48,35){\makebox(0,0)[c]{$\cdot$}}
\put(35,35){\makebox(0,0){$\bullet$}}
\put(55,35){\makebox(0,0){$\bullet$}}
\put(45,15){\makebox(0,0){$\bullet$}}
\put(45,15){\line(1,2){10}}
\put(45,15){\line(-1,2){10}}
\put(45,15){\line(1,0){30}}
\put(72,35){\makebox(0,0)[c]{$\cdot$}}
\put(75,40){\makebox(0,0)[bc]{$\overbrace{\quad\quad}^{a_3-1}$}}
\put(75,35){\makebox(0,0)[c]{$\cdot$}}
\put(78,35){\makebox(0,0)[c]{$\cdot$}}
\put(65,35){\makebox(0,0){$\bullet$}}
\put(85,35){\makebox(0,0){$\bullet$}}
\put(75,15){\makebox(0,0){$\bullet$}}
\put(75,15){\line(1,2){10}}
\put(75,15){\line(-1,2){10}}
\put(105,5){\line(-3,1){30}}
\put(105,5){\line(-1,0){90}}
\put(102,25){\makebox(0,0)[c]{$\cdot$}}
\put(105,30){\makebox(0,0)[bc]{$\overbrace{\quad\quad}^{a_4-1}$}}
\put(105,25){\makebox(0,0)[c]{$\cdot$}}
\put(108,25){\makebox(0,0)[c]{$\cdot$}}
\put(95,25){\makebox(0,0){$\bullet$}}
\put(115,25){\makebox(0,0){$\bullet$}}
\put(105,5){\makebox(0,0){$\bullet$}}
\put(105,5){\line(1,2){10}}
\put(105,5){\line(-1,2){10}}
\end{picture}}

\begin{document}


 \author[T. Kadoi]{KADOI, Tomoe}
 \address[Kadoi]{Department of Mathematical Sciences\\
 Shinshu University\\
 3-1-1 Asahi, Matsumoto-shi, Nagano-ken, 390-8621, Japan.}
 \author[Y. Numata]{NUMATA, Yasuhide}
 \address[Numata]{Department of Mathematical Sciences\\
 Shinshu University\\
 3-1-1 Asahi, Matsumoto-shi, Nagano-ken, 390-8621, Japan.}
 \thanks{This work was supported by JSPS KAKENHI Grant Number 25800009}
\title[Graphs whose Hosoya indices are Pythagorean triples]{On graphs whose Hosoya indices are primitive Pythagorean triples}
\begin{abstract}
 We discuss families of triples of graphs whose Hosoya indices are
 primitive Pythagorean triples.
 Hosoya gave a method to construct such families of caterpillars,
 i.e., 
 trees whose vertices are within distance $1$ of a central path.
 He also pointed out a common structure to the families,
 and conjectured the uniqueness of the structure.
 In this paper, we give an answer to his conjecture.
\end{abstract}
\maketitle

\section{Introduction}
In this paper,
 we consider a non-directed simple graph $G=(V,E)$, i.e., 
a graph whose set of vertices is $V$ and whose set of non-directed edges is $E$
which contains no multiple edges and no self loops.
We can regard $E$ as a subset of $\Set{\ed{v}{w}\subset V|v\neq w}$.
For a graph $G=(V,E)$,
a matching in $G$ is a subset of $E$
such that any two distinct edges do not share any vertices.
The Hosoya index $Z(G)$, or the $Z$-index,
introduced in Hosoya \cite{hosoya1971topological}
is the number of matchings in $G$.
He introduced the graph invariant for his chemical studies,
and developed mathematical theory.
The main objects in this paper 
are triples $(A,B,C)$ of graphs 
such that $(Z(A),Z(B),Z(C))$ is 
\emph{a primitive Pythagorean triple},
i.e.,
a triple $(a,b,c)$ of positive integers
satisfying the following conditions:
\begin{itemize}
 \item $a^2+b^2=c^2$,
 \item $a$, $b$ and $c$ are relatively prime, and
 \item $a$ and $c$ are odd, $b$ is even.
\end{itemize}
In \cite{hosoya2007CCACAA, MR2513595},
Hosoya considered  primitive Pythagorean triples
with consecutive legs,
and gave a family of triples of graphs for them.
In \cite{MR2515294},
he showed that
the Hosoya indices of
a triple of graphs obtained by gluing copies of a graph in some manner 
satisfy the Pythagorean equation.
In \cite{MR2732476},
he introduced a method to construct a triple of graphs
whose Hosoya indices are a primitive Pythagorean triple.
He also conjectured the uniqueness of such families.
In this paper, we give an answer to his conjecture.

This paper is organized as follows:
In Section \ref{sec:def},
we define notation and give examples.
We show  main results in Section \ref{sec:main}.
In the proof, we use some technical lemmas shown in Section \ref{sec:techlemma}.
\section{Definition}
\label{sec:def}
First we define some typical graphs in our discussion.
\begin{definition}
 For positive integers $a_1,\ldots, a_d$,
 we define the \emph{caterpillar} $C(a_1,\ldots,a_d)$ to be the following graph $(V,E)$:
 \begin{align*}
  V_i=&\Set{(i,j)|j=1,\ldots,a_i},\\
  V=&\bigcup_{i=1}^d V_i,\\
  E_i=&\Set{\ed{(i,1)}{(i,j)}| 1< j \leq a_i},\\
  H=&\Set{\ed{(i,1)}{(i+1,1)}|i=1,\ldots,d-1},\\
  E=&H\cup \bigcup_{i=1}^d E_i.
 \end{align*}
\end{definition}
 The graph $(V_i,E_i)$ is the star graph or the complete bipartite graph $K_{1,a_i-1}$.
 The edges in $H$ connects the central vertices $(i,1)$ and
 $(i+1,1)$ of $(V_i,E_i)$ and $(V_{i+1},E_{i+1})$. 
 Hence $C(a_1,\ldots,a_d)$ is 
 \begin{align*}
  \figI.
 \end{align*}
 Note that the caterpillar $C(a_1,\ldots,a_{d-1},a_{d},1)$ is
 isomorphic to the caterpillar $C(a_1,\ldots,a_{d-1},a_{d}+1)$ as graphs.

\begin{definition}
 For positive integers $a_1,\ldots,a_d$ and $s$,
 we define the graph $Q_s(a_1,\ldots,a_d)$ to be 
 the graph obtained from the caterpillar $C(a_1,\ldots,a_d)$
 by adding the edge $\ed{(1,1)}{(s,1)}$.
\end{definition}

\begin{example}
Since the caterpillar $C(a_1,\ldots,a_d)$ is a tree,
the graph $Q_3(a_1,\ldots,a_d)$ has a unique cycle of length $3$.
The graph $Q_3(a_1,\ldots,a_d)$ is 
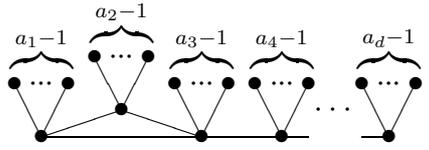
\begin{align*}
 \figII.
\end{align*}
 Note that the graph  $Q_3(a_1,a_2,a_3,\ldots,a_d)$ is
 isomorphic to the graph $Q_3(a_2,a_1,a_3,\ldots,a_d)$  as graphs.
\end{example}
\begin{example}
The graph
$Q_4(a_1,a_2,a_3,a_4)$ has a unique cycle of length $4$.
The graph $Q_4(a_1,a_2,a_3,a_4)$ is
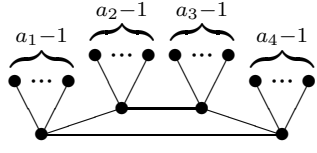
\begin{align*}
 \figIII.
\end{align*}
 Note that 
$Q_4(a_1,a_2,a_3,a_4)$, 
$Q_4(a_2,a_3,a_4,a_1)$ and
$Q_4(a_4,a_3,a_2,a_1)$ 
 are
 isomorphic to one another as graphs.
\end{example}

\begin{definition} 
 For positive integers $a_1,\ldots,a_d$ and $s$,
we define the graph $\Theta_{s}(a_1,\ldots,a_d)$ to be the graph
 obtained from the caterpillar
 $C(a_1,\ldots,a_d)$
 by adding the edge $\ed{(1,1)}{(s,1)}$ and  $\ed{(1,1)}{(d,1)}$.
\end{definition}
\begin{example}
For four positive integers $a_1$, $a_2$, $a_3$ and $a_4$,
the graph $\Theta_{3}(a_1,a_2,a_3,a_4)$ is
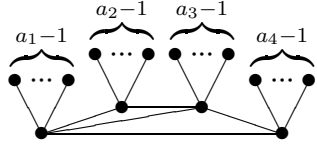
\begin{align*}
 \figIV .
\end{align*}
\end{example}

In this paper, we discuss a common structure to a family of graphs.
To formulate such structures,
we define an operation to make a new graph from two graphs.
\begin{definition}
 Let $G$ and $H$ be graphs.
 For a vertex $v$ in $G$ and a vertex $h$ in $H$,
 we define a \emph{one-point union $G\ops{g}{h} H$ at base points $g$ and $h$}
 to be the 
 graph obtained from the disjoint union of $G$ and $H$
 by contracting the vertices $g$ and $h$.
\end{definition}
If the context makes base points clear,
then we will omit base points as $G \vee H$.

\begin{definition}
 For $i=1,\ldots,d$, 
 we write $C(a_1,\ldots,\dot a_i,\ldots,a_d)$
 to denote the caterpillar $C(a_1,\ldots,a_d)$
 with the base point $(i,1)$ to make a one-point union.
 Similarly,
 for $l$ and $r$ satisfying $l\leq r$,
 we write $C(a_1,\ldots,\dot a_l,\ldots,\dot a_r,\ldots,a_d)$
 to denote
 the caterpillar $C(a_1,\ldots,a_d)$
 with the base point $(l,1)$ to make a one-point union from the left
 and
 with the base point $(r,1)$ to make a one-point union from the right.
 If $l=r$, then we write
 $C(a_1,\ldots,\ddot a_l,\ldots,a_d)$.
 We also define this notation for $Q$
 in the same manner.
\end{definition}
\begin{example}
Let 
\begin{align*}
G&=C(\dot a_1,a_2,\ldots,a_{d-1},\dot a_{d}),\\
H&=C(\dot b_1,b_2,\ldots,b_{e-1},\dot b_{e}).
\end{align*}
In this case, we have
\begin{align*}
 G\vee H = C(\dot a_1,a_2,\ldots,a_{d-1},a_{d}+b_1-1,b_2,\ldots,b_{e-1},\dot b_{e}).
\end{align*}
\end{example}

\begin{definition}
 Let $S_i$ be a graph 
 with the base point $l_i$ to make a one-point union from the left
 and
 with the base point $r_i$ to make a one-point union from the right.
 We say that $(S_1,S_2,S_3)$ is a \emph{symmetric kernel} of 
 a triple $(G_1,G_2,G_3)$ of graphs
 if there exist a graph $G$ and a vertex $g$ of $G$ such that
 \begin{align*}
  G_i =G\ops{g}{l_i} S_i \ops{r_i}{g} G \quad (i=1,2,3).
 \end{align*}
\end{definition}

Next we recall the Hosoya index for a graph.
The Hosoya index is defined as follows.
\begin{definition}
  For a graph $G$,
  we call a set $M$ of edges a \emph{matching} in $G$
  if no pair of edges in $M$ share a vertex.
  Define $p(G,k)$ to be the number of matchings with $k$ edges in $G$.
  We also define 
 the \emph{Hosoya index} $Z(G)$ by $Z(G)=\sum_{k}p(G,k)$.
\end{definition}

\begin{remark}
  If $G'$ is a subgraph of $G$,
  then 
  a matching in $G'$ is a matching in $G$.
  Hence we have $p(G',k)\leq p(G,k)$,
  which implies $Z(G')\leq Z(G)$.
  Since $p(G,k)$ is the number of edges in $G$,
  the Hosoya index $Z(G)$ is greater than the number of edges in $G$.
\end{remark}

The following formulas for $Z(G)$ are known.
\begin{lemma}
If the graph $G$ is the disjoint union of graphs $G_1$ and $G_2$,
then
the graphs $G$, $G_1$ and $G_2$ satisfy the equation
\begin{align*}
 Z(G)=Z(G_1)\cdot Z(G_2).
\end{align*}
\end{lemma}
\begin{lemma}
\label{lemma:removeedge}
Fix an edge $\ed{u}{v}$ of the graph $G$.
Let $G_{\ed{u}{v}}$ be the graph obtained from $G$ by removing the edge $\ed{u}{v}$,
and $G_{u,v}$ the restriction of $G$  to  vertices
other than $u$ and $v$.
The graphs $G$, $G'$ and $G''$ satisfy the equation
\begin{align*}
Z(G)=Z(G_{\ed{u}{v}}) + Z(G_{u,v}).
\end{align*}
\end{lemma} 
Thanks to these two formulas, we can calculate the Hosoya index recursively.
\begin{example}
\label{ex:list:bynum}
The following are the complete list of connected graphs 
with at most three edges:
 \begin{align*}
  Z(C(1)) &=1,&
  Z(C(2))&=2,&
  Z(C(3))&=3,\\
  Z(C(2,2))&=5,&
  Z(C(4))&=4,&
  Z(Q_3(1,1,1))&=4.
 \end{align*}
The following are the complete list of connected graphs 
with four edges:
 \begin{align*}
  Z(C(2,1,2))&=8,&
  Z(C(2,3))&=7,&
  Z(C(5))&=5,\\
  Z(Q_3(1,1,2))&=6,&
  Z(Q_4(1,1,1,1))&=7.
 \end{align*}
The following are the complete list of caterpillar graphs 
with five edges:
 \begin{align*}
  Z(C(2,1,1,2))&=13,&
  Z(C(2,1,3))&=11,&
  Z(C(2,2,2))&=12,&\\
  Z(C(2,4))&=9,&
  Z(C(3,3))&=10,&
  Z(C(6))&=6.
 \end{align*}
The following are the complete list of connected graphs 
with five edges containing $Q_3(1,1,1)$ as a subgraph:
 \begin{align*}
  Z(Q_3(1,1,1;2))&=10,&
  Z(Q_3(1,1,3))&=8,&
  Z(Q_3(2,1,2))&=9,&\\
  Z(\Theta_{3}(1,1,1,1))&=8. &
 \end{align*}
The following are the complete list of the other connected graphs 
with five edges:
 \begin{align*}
  Z(Q_4(1,1,1,2))&=10,&
  Z(Q_5(1,1,1,1,1))&=11.
 \end{align*}
See also Section \ref{subsec:tableofindex}.
\end{example}

Finally we define notation for continued fractions and continuants.
\begin{definition}
For positive integers $a_1,\ldots,a_d$,
we define a rational number $[a_1,\ldots,a_d]$ by
\begin{align*}
[a_1,\ldots,a_d]
=
a_1+\cfrac{1}{
a_2+\cfrac{1}{
\ddots\phantom{+}\cfrac{\ddots}{
a_{d-1}+\cfrac{1}{
a_d
}
}
}
}.
\end{align*}
\end{definition}
Note that $[a_1,\ldots,a_{d-1},a_d,1]=[a_1,\ldots,a_{d-1},a_d+1]$.
\begin{definition}
 We define the polynomials
 $K_d$ in variables $x_1,x_2,\ldots$ 
 by 
 \begin{align*}
  K_0  &= 1,\\ 
  K_1  &=x_1,\\
  K_d  &=x_dK_{d-1}+K_{d-2}.
 \end{align*}
 The polynomial $K_d$ is called the \emph{continuant}.
 For positive integers $a_1,\ldots,a_d$,
 we write $K(a_1,\ldots,a_d)$
 to denote the integer obtained by substituting
 $a_i$ for $x_i$ in $K_d$.
\end{definition}
It is known that 
$K(a_1,\ldots,a_d)$ and $K(a_2,\ldots,a_d)$ are relatively prime
integers satisfying
\begin{align*}
[a_1,\ldots,a_d]=\frac{K(a_1,\ldots,a_d)}{K(a_2,\ldots,a_d)}
\end{align*}
for positive integers $a_1,\ldots,a_d$.
In \cite{MR2416190}, Hosoya showed that
\begin{align*}
K(a_1,\ldots,a_d)=Z(C(a_1,\ldots,a_d)).
\end{align*}
Since $K(a_1,\ldots,a_d)=K(a_d,\ldots,a_1)$,
it also follows that
\begin{align*}
 K(a_1,a_2,\ldots,a_d)&=a_1K(a_2,\ldots,a_d)+K(a_3,\ldots,a_d).
\end{align*} 

\begin{lemma}
 \label{lemma:k:induction}
 If positive integers $a_1,\ldots,a_d$, $b_1,\ldots,b_e$ satisfy
 \begin{align*}
  K(a_1,\ldots,a_d)&=K(b_1,\ldots,b_e),\\
  K(a_2,\ldots,a_d)&=K(b_2,\ldots,b_e),\\
  a_d,b_e&>1,
 \end{align*}
 then 
 $d=e$ and $(a_1,\ldots,a_d)=(b_1,\ldots,b_e)$.
\end{lemma}
\begin{proof}
 First consider the case where $e=1$.
 In this case, 
 $K(b_1,\ldots,b_e)=K(b_1)=b_1$
 and 
 $K(b_2,\ldots,b_e)=K_0=1$.
 If $d\geq 2$, 
 then we have 
\begin{align*}
K(a_2,\ldots,a_d)\geq K(a_d)=a_d>1,  
\end{align*}
 which contradicts the condition.
 If $d=1$, then
 we have $K(a_1)=a_1$, which implies $a_1=b_1$.

 
 We consider the case where $d \geq e\geq 2$.
 Let 
\begin{align*}
m=K(a_1,\ldots,a_d)=K(b_1,\ldots,b_e),\\ 
n=K(a_2,\ldots,a_d)=K(b_2,\ldots,b_e).
\end{align*}
 It follows  by the definition of continuants that
 \begin{align*}
  m&=K(a_1,\ldots,a_d)=a_1K(a_2,\ldots,a_d)+K(a_3,\ldots,a_d)\\
  &=a_1n+K(a_3,\ldots,a_d),\\
  m&=K(b_1,\ldots,b_e)=b_1K(b_2,\ldots,b_e)+K(b_3,\ldots,b_e)\\
  &=b_1n+K(b_3,\ldots,b_e).
 \end{align*}
 Hence we have
 \begin{align*}
  (b_1-a_1)n=K(a_3,\ldots,a_d)-K(b_3,\ldots,b_e).
 \end{align*}
 Since both $K(a_3,\ldots,a_d)$ and $K(b_3,\ldots,b_e)$
 are less than 
 \begin{align*}
  n=K(a_2,\ldots,a_d)=K(b_2,\ldots,b_e),
 \end{align*}
 we have the inequality
 \begin{align*}
  -n<K(a_3,\ldots,a_d)-K(b_3,\ldots,b_e)< n.
 \end{align*}
 Hence $b_1-a_1=0$.
 It also follows that
 \begin{align*}
  K(a_3,\ldots,a_d)-K(b_3,\ldots,b_e)=0. 
 \end{align*}
 Hence we have the lemma by induction.
\end{proof}

Lemma \ref{lemma:k:induction} implies the following lemma.
\begin{lemma}
 \label{lemma:cat:induction}
 Let $G=C(a_1,\ldots,a_d)$ and $G'=C(b_1,\ldots,b_e)$.
 Assume that $d\geq e\geq 2$.
 If  \begin{align*}
  Z(C(a_1,\ldots,a_d))&=Z(C(b_1,\ldots,b_e)),\\
  Z(C(a_2,\ldots,a_d))&=Z(C(b_2,\ldots,b_e)),
 \end{align*}
 then the caterpillar $G$ is isomorphic to
 the caterpillar $G'$ 
 as graphs, and $a_i=b_i$ for each $i=1,\ldots, e-1$. 
\end{lemma}
\begin{proof}
 If $a_d$ and $b_e$ are greater than $1$,
 then it follows from Lemma \ref{lemma:k:induction}
 that $G$ is isomorphic to $G'$ as graphs.
 Consider the case where $b_e=1$.
 If $e> 2$, then
 we have 
 \begin{align*}
 C(b_1,\ldots,b_{e-1},1)&=C(b_1,\ldots,b_{d-1}+1),\\
 C(b_2,\ldots,b_{e-1},1)&=C(b_2,\ldots,b_{d-1}+1).
 \end{align*}
 Hence we may assume that $b_e>1$.
 Similarly we may assume that $a_d>1$.
 Hence it follows from Lemma \ref{lemma:k:induction}
 that $G$ is isomorphic to $G'$ as graphs.
 If $e=2$, then  $Z(C(a_2,\ldots,a_d))=Z(C(b_2))=1$.
 Hence the graph $G$ is the caterpillar $C(a_1,1)$. Since $Z(C(a_1,1))=Z(C(b_1,1))$,
 we have
 $a_1+1=b_1+1$, which implies $a_1=b_1$.
\end{proof}
\section{Main results}
\label{sec:main}
We consider a family of triples $(A_{m,n},B_{m,n},C_{m,n})$ of
graphs which gives
primitive Pythagorean triples as $(Z(A_{m,n}),Z(B_{m,n}),Z(C_{m,n}))$,
and the common symmetric kernel to the family.
In other words, we consider the triple $(A,B,C)$ of graphs with base points
and the family of graphs $G_{m,n}$ with base points satisfying
\begin{align*}
 A_{m,n} &= G_{m,n}\vee A \vee G_{m,n},\\
 B_{m,n} &= G_{m,n}\vee B \vee G_{m,n},\\
 C_{m,n} &= G_{m,n}\vee C \vee G_{m,n}.
\end{align*}

Let $\PPP$ be the set of the pair $(m,n)$ of positive integers 
satisfying the following:
\begin{itemize}
 \item $m$ and $n$ are relatively prime,
 \item $m$ and $n$ have opposite parity,
 \item $m>n$.
\end{itemize}
It is known that
there exists a bijection $\varphi$ from $\PPP$ to 
the set of primitive Pythagorean triples
defined by $\varphi(m,n)=(m^2-n^2,2mn,m^2+n^2)$.

Define $\PPP_1$ and $\PPP_2$ by
\begin{align*}
\PPP_1&=\Set{(m,n)\in\PPP| \frac{m}{n} \geq 2}\\
\PPP_2&=\Set{(m,n)\in\PPP|1< \frac{m}{n} <2}. 
\end{align*}
In \cite{MR2732476},
Hosoya developed a method,
based on Euclidean algorithm,
to construct
a triple of caterpillars 
whose Hosoya indices are
primitive Pythagorean triples
corresponding to $\PPP_1$.
He also gave one for $\PPP_2$.
Here we show  essentially the same propositions as 
his methods.
\begin{proposition}
 \label{prop:exist:1}
 Let $(m,n)\in \PPP_1$ and positive integers $a_1,\ldots,a_d$
 satisfy $m/n=[a_1,\ldots,a_d]$.
 Note that $a_1\geq 2$.
 Define $G^{(1)}_{m,n}$, $A_{m,n}$, $B_{m,n}$, $C_{m,n}$
 by
 \begin{align*}
  G^{(1)}_{m,n}&=C(\dot{a_1-1},a_2,\ldots,a_d),\\
  A_{m,n}  &=G^{(1)}_{m,n}\vee C(\dot1,1,\dot1)\vee G^{(1)}_{m,n},\\
  B_{m,n}  &=G^{(1)}_{m,n}\vee C(\ddot4)\vee G^{(1)}_{m,n},\\
  C_{m,n}  &=G^{(1)}_{m,n}\vee C(\dot2,\dot2)\vee G^{(1)}_{m,n}.
 \end{align*}
 The triple $(Z(A_{m,n}),Z(B_{m,n}),Z(C_{m,n}))$ of Hosoya indices 
 is the primitive Pythagorean triple $(m^2-n^2,2mn,m^2+n^2)$.
\end{proposition}
\begin{proof}
 Since $m$ and $n$ are relatively prime, 
 we have
 \begin{align*}
  Z(C(a_1,\ldots,a_d)))&=
  K(a_1,\ldots,a_d)=m,\\
  Z(C(a_2,\ldots,a_d)))&=
  K(a_2,\ldots,a_d)=n.
 \end{align*}
 Since $G^{(1)}_{m,n}=C(\dot{a_1-1},a_2,\ldots,a_d)$,
 we obtain
 \begin{align*}
  A_{m,n}  &=G^{(1)}_{m,n}\vee C(\dot1,1,\dot1)\vee G^{(1)}_{m,n}\\
  &=C(a_d,a_{d-1},\ldots,a_2,a_1-1,1,a_1-1,a_2,\ldots,a_{d-1},a_d),\\
  B_{m,n}  &=G^{(1)}_{m,n}\vee C(\ddot4)\vee G^{(1)}_{m,n}\\
  &=C(a_d,a_{d-1},\ldots,a_3,a_2,2a_1,a_2,a_3,\ldots,a_{d-1},a_d),\\
  C_{m,n}  &=G^{(1)}_{m,n}\vee C(\dot2,\dot2)\vee G^{(1)}_{m,n}\\
  &=C(a_d,a_{d-1},\ldots,a_3,a_2,a_1,a_1,a_2,a_3,\ldots,a_{d-1},a_d).
 \end{align*}
 Applying Lemma \ref{lemma:removeedge} to the central edge
 $\ed{(d,1)}{(d+1,1)}$  of $C_{m,n}$,
 we obtain
 \begin{align*}
  &Z(C_{m,n})\\  
  =&Z(C(a_d,a_{d-1},\ldots,a_3,a_2,a_1,a_1,a_2,a_3,\ldots,a_{d-1},a_d))\\
  =&Z(C(a_d,a_{d-1},\ldots,a_3,a_2,a_1)\cdot Z(C(a_1,a_2,a_3,\ldots,a_{d-1},a_d)\\
  &+Z(C(a_d,a_{d-1},\ldots,a_3,a_2))\cdot Z(C(a_2,a_3,\ldots,a_{d-1},a_d)\\
  =&m^2+n^2.
 \end{align*}
Similarly,
since 
\begin{align*}
& Z(C(2a_1,a_2,a_3,\ldots,a_d))+ Z(C(a_3,\ldots,a_d))\\
&= 2a_1Z(C(a_2,a_3,\ldots,a_d))+ Z(C(a_3,\ldots,a_d))+ Z(C(a_3,\ldots,a_d))\\
&= 2(a_1Z(C(a_2,a_3,\ldots,a_d))+ Z(C(a_3,\ldots,a_d)))\\
&= 2Z(C(a_1,a_3,\ldots,a_d))=2m,
\end{align*}
 we obtain
\begin{align*}
 &Z(B_{m,n}) \\
=&Z(C(a_d,a_{d-1},\ldots,a_3,a_2,2a_1,a_2,a_3,\ldots,a_{d-1},a_d))\\
=&Z(C(a_d,a_{d-1},\ldots,a_3,a_2,2a_1))\cdot Z(C(a_2,a_3,\ldots,a_{d-1},a_d))\\
&+Z(C(a_d,a_{d-1},\ldots,a_3,a_2))\cdot Z(C(a_3,\ldots,a_{d-1},a_d))\\
=&Z(C(2a_1,a_2,\ldots,a_d))\cdot n+n\cdot Z(C(a_3,\ldots,a_d))\\
=&n(Z(C(2a_1,a_2,\ldots,a_d))+ Z(C(a_3,\ldots,a_d)))\\
=&2mn.
\end{align*}
 Since $K(a_1,\ldots,a_d)=a_1K(a_2,\ldots,a_d)+K(a_3,\ldots,a_d)$,
 we have
\begin{align*}
 K(a_1-1,\ldots,a_d)
 &= (a_1-1)K(a_2,\ldots,a_d)+K(a_3,\ldots,a_d)\\
 &= a_1K(a_2,\ldots,a_d)+K(a_3,\ldots,a_d)-K(a_2,\ldots,a_d)\\
 &= K(a_1,\ldots,a_d)-K(a_2,\ldots,a_d)\\
 &=m-n.
\end{align*}
 Hence it follows that
 \begin{align*}
  &Z(A_{m,n})  \\
  =&Z(C(a_d,a_{d-1},\ldots,a_2,a_1-1,1,a_1-1,a_2,\ldots,a_{d-1},a_d))\\
  =&
  Z(C(a_d,a_{d-1},\ldots,a_2,a_1-1))\cdot Z(C(1,a_1-1,a_2,\ldots,a_{d-1},a_d))\\
  &+Z(C(a_d,a_{d-1},\ldots,a_2))\cdot Z(C(a_1-1,a_2,\ldots,a_{d-1},a_d))\\
  =&
  (m-n)\cdot Z(C(1,a_1-1,a_2,\ldots,a_d))+Z(C(a_2,\ldots,a_d))\cdot (m-n)\\
  =& (m-n) ( Z(C(a_1,a_2,\ldots,a_d))+Z(C(a_2,\ldots,a_d)))\\
  =&(m-n)(m+n)=m^2+n^2.
\qedhere \end{align*}\end{proof}
\begin{proposition}
 Let $(m,n)\in \PPP_2$ and positive integers $a_1,\ldots,a_d$
 satisfy $m/n=[a_1,\ldots,a_d]$.
 Define $G^{(2)}_{m,n}$, $A_{m,n}$, $B_{m,n}$, $C_{m,n}$
 by
 \begin{align*}
  G^{(2)}_{m,n}&=C(\dot a_2,a_3,\ldots,a_d),\\
  A_{m,n}&=G^{(2)}_{m,n}\vee C(\ddot3)\vee G^{(2)}_{m,n},\\
  B_{m,n}&=G^{(2)}_{m,n}\vee C(\dot1,2,\dot1)\vee G^{(2)}_{m,n},\\
  C_{m,n}&=G^{(2)}_{m,n}\vee C(\dot1,1,1,\dot1)\vee G^{(2)}_{m,n}.
 \end{align*}
 The triple $(Z(A_{m,n}),Z(B_{m,n}),Z(C_{m,n}))$ of Hosoya indices 
 is the primitive Pythagorean triple $(m^2-n^2,2mn,m^2+n^2)$.
\end{proposition}
\begin{proof} 
Similarly to the case where $(m,n)\in\PPP_1$,
 we have
\begin{align*}
  Z(C(a_1,\ldots,a_d)))&=m,\\
  Z(C(a_2,\ldots,a_d)))&=n. 
\end{align*}
Since $G^{(2)}_{m,n}=C(\dot a_2,a_3,\ldots,a_d)$ and $a_1=1$,
we have
\begin{align*}
 A_{m,n}
 &=G^{(2)}_{m,n}\vee C(\ddot3)\vee G^{(2)}_{m,n}\\
 &=C(a_d,a_{d-1},\ldots,a_3,2a_2+1,a_3,\ldots,a_{d-1},a_d),\\
 B_{m,n}
 &=G^{(2)}_{m,n}\vee C(\dot1,2,\dot1)\vee G^{(2)}_{m,n}\\
 &=C(a_d,a_{d-1},\ldots,a_3,a_2,2,a_2,a_3,\ldots,a_{d-1},a_d)\\
 &=C(a_d,a_{d-1},\ldots,a_3,a_2,2a_1,a_2,a_3,\ldots,a_{d-1},a_d),\\
 C_{m,n}
 &=G^{(2)}_{m,n}\vee C(\dot1,1,1,\dot1)\vee G^{(2)}_{m,n}\\
 &=C(a_d,a_{d-1},\ldots,a_3,a_2,1,1,a_2,a_3,\ldots,a_{d-1},a_d)\\
 &=C(a_d,a_{d-1},\ldots,a_3,a_2,a_1,a_1,a_2,a_3,\ldots,a_{d-1},a_d).
\end{align*}
Hence, similarly to the case where $(m,n)\in\PPP_1$,
we obtain $Z(B_{m,n})=2mn$ and $Z(C_{m,n})=m^2+n^2$.
Since 
\begin{align*}
 Z(C(a_1,\ldots,a_d))&= a_1Z(C(a_2,\ldots,a_d))+Z(C(a_3,\ldots,a_d))\\
 &= Z(C(a_2,\ldots,a_d))+Z(C(a_3,\ldots,a_d)),
\end{align*}
we have 
\begin{align*}
 Z(C(a_3,\ldots,a_d))&=Z(C(a_1,\ldots,a_d))-Z(C(a_2,\ldots,a_d))\\
 &=m-n.
\end{align*}
Applying Lemma \ref{lemma:removeedge} to 
central edges $\ed{(d-1,1)}{(d-1,j)}$ in $A_{m,n}$ 
for $j=2,3,\ldots,2a_2+1$,
we obtain
\begin{align*}
 &Z(A_{m,n}) \\
=&Z(C(a_d,a_{d-1},\ldots,a_3,2a_2+1,a_3,\ldots,a_{d-1},a_d))\\
=&Z(C(a_d,a_{d-1},\ldots,a_3,1,a_3,\ldots,a_{d-1},a_d))+2a_2Z(C(a_3,\ldots,a_{d-1},a_d))^2\\
=&Z(C(a_d,a_{d-1},\ldots,a_3,1,a_3,\ldots,a_{d-1},a_d))+2a_2(m-n)^2.
\end{align*}
Since 
\begin{align*}
&Z(C(a_d,a_{d-1},\ldots,a_3,1,a_3,\ldots,a_{d-1},a_d))\\
=&Z(C(a_d,a_{d-1},\ldots,a_3,1))\cdot Z(C(a_3,\ldots,a_{d-1},a_d))\\
&+Z(C(a_d,a_{d-1},\ldots,a_3))\cdot Z(C(a_4,\ldots,a_{d-1},a_d))\\
=&Z(C(1,a_3,\ldots,a_d))\cdot (m-n)
+(m-n)\cdot Z(C(a_4,\ldots,a_d))\\
=& 
(m-n)(Z(C(a_d,a_{d-1},\ldots,a_3,1))+ Z(C(a_4,\ldots,a_{d-1},a_d)),
\end{align*}
we obtain
\begin{align*}
 &Z(A_{m,n}) \\
=&Z(C(a_d,a_{d-1},\ldots,a_3,1,a_3,\ldots,a_{d-1},a_d))
+2a_2(m-n)^2\\
=& (m-n)(Z(C(1,a_3,\ldots,a_d))+
 Z(C(a_4,\ldots,a_{d-1},a_d))+2a_2(m-n)).
\end{align*}
Since $Z(C(1,a_3,\ldots,a_d))=Z(C(a_3,\ldots,a_d))+Z(C(a_4,\ldots,a_d))$,
we have
\begin{align*}
&Z(C(1,a_3,\ldots,a_d))+ Z(C(a_4,\ldots,a_d))+2a_2(m-n)\\
=&Z(C(a_3,\ldots,a_d))+2(Z(C(a_4,\ldots,a_d)+a_2(m-n)).
\end{align*}
Since  $Z(C(a_3,\ldots,a_d))=m-n$,
we have
$Z(C(a_4,\ldots,a_d)+a_2(m-n))=Z(C(a_2,\ldots,a_d))$.
Hence we obtain
\begin{align*}
  Z(A_{m,n})
 &=
(m-n)(Z(C(a_3,\ldots,a_d))+2Z(C(a_2,\ldots,a_d)))\\
 &=
(m-n)(Z(C(a_1,\ldots,a_d))+Z(C(a_2,\ldots,a_d)))\\
 &=(m-n)(m+n)=m^2-n^2.
\qedhere\end{align*}\end{proof}
These families for $\PPP_1$ and $\PPP_2$
have common symmetric kernels
\begin{align*}
 (C(\dot1,1,\dot1), C(\ddot4), C(\dot2,\dot2))
\end{align*}
and
\begin{align*}
(C(\ddot3), C(\dot1,2,\dot1),C(\dot1,1,1,\dot1)),
\end{align*}
respectively.
Hosoya also conjectured the uniqueness of 
such structure.
Main results gives an answer to his conjecture.
First we show our main result for $\PPP_1$.
\begin{theorem}
 \label{thm:main:1}
 Let $\Set{(A_{m,n}, B_{m,n}, C_{m,n})}_{(m,n)\in \PPP_1}$
 be a family of triples of connected graphs
 satisfying 
\begin{align*}
 Z(A_{m,n})&=m^2-n^2,\\
 Z(B_{m,n})&=2mn,\\
 Z(C_{m,n})&=m^2+n^2. 
\end{align*}
 If the family $\Set{(A_{m,n}, B_{m,n}, C_{m,n})}_{(m,n)\in \PPP_1}$
 has the common symmetric kernel $(A^{(1)},B^{(1)},C^{(1)})$,
 then 
 \begin{align*}
  (A^{(1)},B^{(1)},C^{(1)})=(C(\dot1,1,\dot1),C(\ddot4),C(\dot2,\dot2)).
 \end{align*}
\end{theorem}

\begin{proof}
 We find the graphs $A^{(1)}$, $B^{(1)}$, $C^{(1)}$ and $G_{m,n}$
 so that
 \begin{align*}
  Z(G_{m,n} \vee A^{(1)} \vee G_{m,n}) &= m^2-n^2,\\
  Z(G_{m,n} \vee B^{(1)} \vee G_{m,n}) &= 2mn, \\
  Z(G_{m,n} \vee C^{(1)} \vee G_{m,n}) &= m^2+n^2
 \end{align*}
 for $(m,n)\in\Set{(2,1),(4,1),(5,2)}$.
Since $Z(G_{2,1}\vee A^{(1)} \vee G_{2,1})=3$,
the graph $G_{2,1}\vee A^{(1)} \vee G_{2,1}$
is the caterpillar $C(3)$. 
By Lemma \ref{lemma:B1},
the graph $G_{2,1} \vee B^{(1)} \vee G_{2,1}$ is the caterpillar $C(4)$.
Hence the graph $G_{2,1}$ is
either $C(\dot1)$ or $C(\dot2)$.

If $G_{2,1}=C(\dot2)$, then 
the graphs $A^{(1)}$ and $B^{(1)}$ are
$C(\ddot1)$ and $C(\ddot2)$, respectively.
Since $Z(G_{4,1}\vee B^{(1)} \vee G_{4,1})=8$,
it follow from the table in Lemma \ref{lemma:table} that
the graph $G_{4,1}$ is the caterpillar $C(\dot4)$.
Hence we obtain $Z(G_{4,1}\vee A^{(1)} \vee G_{4,1})=Z(7)=7$,
which contradicts $Z(G_{4,1}\vee A^{(1)} \vee G_{4,1})=15$.

If $G_{2,1}=C(\dot1)$, 
then 
the graphs $A^{(1)}$ and $B^{(1)}$ are
$C(3)$ and $C(4)$,
respectively.
Since $Z(G_{4,1}\vee B^{(1)} \vee G_{4,1})=8$,
the graph $G_{4,1}\vee B^{(1)} \vee G_{4,1}$ is the caterpillar $C(8)$.
Hence 
the graphs $B^{(1)}$ and $G_{4,1}$
are
$C(\ddot4)$ and $C(\dot3)$,
respectively.
It follows from direct calculation that
\begin{align*}
Z(G_{4,1}\vee A^{(1)\vee G_{4,1}})
=\begin{cases}
  Z(C(3,4))=13 &(A^{(1)}=C(\dot1,\dot2)) \\
  Z(C(5,2))=14&(A^{(1)}=C(\ddot1,2)) \\
  Z(C(7))=7&(A^{(1)}=C(\ddot3)).
 \end{cases}
\end{align*}
Hence the graph  $A^{(1)}$ is $C(\dot1,1,\dot1)$.
Since $Z(C^{(1)})=Z(G_{2,1}\vee C^{(1)} \vee G_{2,1})=5$,
the graph $C^{(1)}$ is either $C(2,2)$ or $C(5)$.
By Lemmas \ref{lemma:C1a} and  \ref{lemma:C1b},
candidates of $C^{(1)}$ are 
$C(\dot2,\dot2)$ and $C(\ddot1,1,2)$.
By Lemma \ref{lemma:C1c},
if $C^{(1)}=C(\ddot1,1,2)$, then
we can not construct $G_{5,2}$.
Hence the graph $C^{(1)}$ is $C(\dot2,\dot2)$.
Therefore
$(A^{(1)},B^{(1)},C^{(1)})=(C(\dot1,1,\dot1),C(\ddot4),C(\dot2,\dot2))$.
\end{proof}

Next we show our main result for $\PPP_2$.
\begin{theorem}
 \label{thm:main:2}
 Let $\Set{(A_{m,n}, B_{m,n}, C_{m,n})}_{(m,n)\in \PPP_2}$
 be a family of triples of connected graphs
 satisfying 
\begin{align*}
 Z(A_{m,n})&=m^2-n^2, \\
 Z(B_{m,n})&=2mn, \\
 Z(C_{m,n})&=m^2+n^2.
\end{align*}
 If the family $\Set{(A_{m,n}, B_{m,n}, C_{m,n})}_{(m,n)\in \PPP_2}$
 has the common symmetric kernel $(A^{(2)},B^{(2)},C^{(2)})$,
 then 
 \begin{align*}
  (A^{(2)},B^{(2)},C^{(2)})=(C(\ddot3),C(\dot1,3,\dot1),C(\dot1,1,1,\dot1)).
 \end{align*}
\end{theorem}
\begin{proof}
 We find the graphs $A^{(2)}$, $B^{(2)}$, $C^{(2)}$ and $G_{m,n}$
 so that
 \begin{align*}
  Z(G_{m,n} \vee A^{(2)} \vee G_{m,n}) &= m^2-n^2,\\
  Z(G_{m,n} \vee B^{(2)} \vee G_{m,n}) &= 2mn, \\
  Z(G_{m,n} \vee C^{(2)} \vee G_{m,n}) &= m^2+n^2
 \end{align*}
 for $(m,n)\in\Set{(3,2),(4,3),(7,4)}$.
 It follows from Lemma \ref{lemma:A2}
 that the graph $G_{4,3}\vee A^{(2)}\vee G_{4,3}$ is the caterpillar $C(7)$, 
 which implies
 \begin{align*}
  (A^{(2)}, G_{3,2}, G_{4,3}) \in 
   \Set{
  \begin{array}{c}
  (C(\ddot5),C(\dot1),C(\dot2)),\\
   (C(\ddot3),C(\dot2),C(\dot3)),\\
   (C(\ddot1),C(\dot3),C(\dot4))   
  \end{array}
  }.
 \end{align*}
 If $(A^{(2)}, G_{3,2}, G_{4,3})=(C(\ddot5),C(\dot1),C(\dot2))$,
 then we have 
 \begin{align*}
  Z(B^{(2)})=Z(G_{3,2} \vee B^{(2)} \vee G_{3,2})=12. 
 \end{align*}
 By Lemma \ref{lemma:caseA},
 we have
 $Z(G_{7,4}\vee B^{(2)} \vee G_{7,4})\neq 56$
 for any $(G_{7,4},B^{(2)})$ such that
 \begin{align*}
  Z(G_{7,4}\vee A^{(2)} \vee G_{7,4})&=33\\
  Z(G_{4,3}\vee B^{(2)}\vee G_{4,3})&= 24.  
 \end{align*}
 Hence $(A^{(2)}, G_{3,2}, G_{4,3})=(C(5),C(1),C(2))$ does not
 satisfy the condition.
 By Lemma \ref{lemma:caseC},
 for $(A^{(2)}, G_{3,2}, G_{4,3})=(C(\ddot1),C(\dot3),C(\dot4))$,
 we have  $Z(G_{4,3}\vee B^{(2)}\vee G_{4,3})\neq 24$.
 Hence $(A^{(2)}, G_{3,2}, G_{4,3})=(C(\ddot1),C(\dot3),C(\dot4))$
 does not satisfy the condition.
 By Lemma \ref{lemma:caseB},
 if
 $(A^{(2)}, G_{3,2}, G_{4,3})=(C(\ddot3),C(\dot2),C(\dot3))$
 and
 $B^{(2)}\neq C(\dot1,2,\dot1)$,
 then  
 \begin{align*}
  Z(G_{4,3}\vee B^{(2)}\vee G_{4,3})\neq 24.
 \end{align*}
 Hence we may consider only the case 
 where  
 \begin{align*}
  (A^{(2)},B^{(2)} ,G_{3,2}, G_{4,3})=(C(\ddot3),C(\dot1,2,\dot1),C(\dot2),C(\dot3)).  
 \end{align*}
 In this case,
 it follows from Lemma \ref{lemma:G74} that $G_{7,4}=C(\dot1,3)$.
 We also obtain $C^{(2)}=C(\dot1,1,1,\dot1)$ by Lemma \ref{lemma:C2}.
 Therefore we have
 \begin{align*}
 (A^{(2)},B^{(2)},C^{(2)})&=(C(\ddot3),C(\dot1,2,\dot1),C(\dot1,1,1,\dot1)).
 \qedhere\end{align*}\end{proof}

Theorems \ref{thm:main:1} and \ref{thm:main:2} 
imply
the uniqueness of the symmetric kernels
 $(A^{(1)},B^{(1)},C^{(1)})$ and
 $(A^{(2)},B^{(2)},C^{(2)})$.
However $G_{m,n}$ are not unique.
\begin{example}
Let 
$G=C(\dot2,3)$ and
$G'=Q_4(\dot1,1,1,1)$.
In this case, the Hosoya indices of $G$ and $G'$
are equal to each other.
Moreover the Hosoya indices of 
the restrictions $G$ and $G'$ to vertices other than the base points
are also equal to each other.
Hence we have $Z(A)=Z(A')$, $Z(B)=Z(B')$ and $Z(C)=Z(C')$,
where 
\begin{align*}
A&=G\vee C(\ddot3)\vee G,&
A'&=G'\vee C(\ddot3)\vee G',\\
B&=G\vee C(\dot1,2,\dot1)\vee G,&
B'&=G'\vee C(\dot1,2,\dot1)\vee G', \\
C&=G\vee C(\dot1,1,1,\dot1)\vee G,&
C'&=G'\vee C(\dot1,1,1,\dot1)\vee G'. 
\end{align*}
Since $[1,2,3]$ is equal to $10/7$,
the both triples $(Z(A),Z(B),Z(C))$ and $(Z(A'),Z(B'),Z(C'))$
are equal to the primitive Pythagorean triple for $(10,7)\in\PPP_2$. 
\end{example}
If we consider only caterpillars,
then 
we have the uniqueness of $G_{m,n}$.
First we show the theorem for $(m,n)\in \PPP_1$.
\begin{theorem}
Define the caterpillars $G$, $A$, $B$ and $C$ by
 \begin{align*}
G&=C(\dot\alpha_1,\alpha_2,\ldots,\alpha_\delta),\\
A&=G\vee C(\dot1,1,\dot1)\vee G,\\
B&=G\vee C(\ddot4)\vee G, \\
C&=G\vee C(\dot2,\dot2)\vee G).  
 \end{align*}
 If 
 the triple $(Z(A),Z(B),Z(C))$ 
 is the primitive Pythagorean triple $(m^2-n^2,2mn,m^2+n^2)$
 for $(m,n)\in \PPP_1$,
 then we have
\begin{align*}
 \frac{m}{n}=[\alpha_1+1,\alpha_2\ldots,\alpha_d].
\end{align*}
\end{theorem}
\begin{proof}
 Let $\mu=Z(C(\alpha_1,\alpha_2,\ldots,\alpha_\delta))$,
 $\nu=Z(C(\alpha_2,\ldots,\alpha_\delta))$.
 Similarly to the proof of Proposition \ref{prop:exist:1},
 we obtain $Z(A)=\mu^2-\nu^2$, $Z(B)=2\mu\nu$ and $Z(C)=\mu^2+\nu^2$.
 If both $\mu$ and $\nu$ are odd or both are even,
 then it follows that $Z(A)$ and $Z(C)$ are even,
 which contradicts the assumption that 
 the triple $(Z(A),Z(B),Z(C))$ 
 is the primitive Pythagorean triple.
 Moreover,
 since $\mu=K(\alpha_1,\ldots,\alpha_\delta)$ and  
 $\nu=K(\alpha_2,\ldots,\alpha_d)$,
 the positive integers $\mu$ and $\nu$ are relatively prime and satisfy $\mu>\nu$.
 Hence $(\mu,\nu)\in \PPP$.
 It follows from 
 the bijectivity of $\varphi$ that
 $\mu=m$ and that $\nu=n$.
 Let $C(a_1-1,a_2,\ldots,a_d)$ be the caterpillar
 $G^{(1)}_{m,n}$ in Proposition \ref{prop:exist:1}.
 Both caterpillars $C(a_1-1,a_2,\ldots,a_d)$ and 
 $C(\alpha_1,\ldots,\alpha_\delta)$
 satisfy  the equations
 \begin{align*}
  Z(C(a_1-1,a_2,\ldots,a_d))&=C(\alpha_1,\ldots,\alpha_\delta)=m,\\  
  Z(C(a_2,\ldots,a_d))&=C(\alpha_2,\ldots,\alpha_\delta)=n.
 \end{align*}
 Therefore it follows from Lemma \ref{lemma:cat:induction} that
 \begin{align*}
  \frac{m}{n}&=[a_1,a_2,\ldots,a_d]=[\alpha_1+1,\alpha_2,\ldots,\alpha_\delta].
\qedhere\end{align*}\end{proof}
Similarly we can show the theorem for $(m,n)\in \PPP_2$.
\begin{theorem}
Define the caterpillars $G$, $A$, $B$ and $C$ by
\begin{align*}
 G&=C(\dot \alpha_1,\alpha_2,\ldots,\alpha_\delta), \\
 A&=G\vee C(\ddot3)\vee G,\\
 B&=G\vee C(\dot1,2,\dot1)\vee G, \\
 C&=G\vee C(\dot1,1,1,\dot1)\vee G.
\end{align*}
 If 
 the triple $(Z(A),Z(B),Z(C))$ 
 is the primitive Pythagorean triple $(m^2-n^2,2mn,m^2+n^2)$
 for $(m,n)\in \PPP_2$,
 then 
 \begin{align*}
  \frac{m}{n}=[1,\alpha_1,\ldots,\alpha_\delta].
 \end{align*}
\end{theorem}


\section{Technical lemmas}
\label{sec:techlemma}
In this section, we show lemmas 
on the Hosoya indices,
which is used in the proof of main results.

\subsection{Tables of Hosoya indices}
\label{subsec:tableofindex}
Here we calculate Hosoya indices for small graphs.
Hosoya indices for all graphs with at most five edges
are in Example \ref{ex:list:bynum}.
The following are the complete list of caterpillars $G$
with six edges such that $Z(G)\leq 13$:
 \begin{align*}
  Z(C(7))&=7,&
  Z(C(2,5))&=11,&
  Z(C(3,4))&=13.
 \end{align*}
The following are the complete list of the other connected graphs $G$
with six edges such that $Z(G)\leq 13$:
 \begin{align*}
  Z(Q_3(1,1,4))&= 10,&
  Z(Q_3(1,2,3))&= 12,&\\
  Z(\Theta_{3}(1,1,2,1))&=11,&
  Z(\Theta_{3}(1,2,1,1))&=12,&\\
  Z(\Theta_{3}(1,1,1,1,1))&=13,&
  Z(Q_4(1,1,1,3))&=13,&  \\
  Z(K_{4})&=10,&
  Z(K_{2,3})&=13,\\
  Z(Q_3(1,1,\dot 1)\vee Q_3(1,1,\dot 1))&= 12,&
 \end{align*}
 where $K_4$ stands for the complete graph and
 $K_{2,3}$ stands for the complete bipartite graph.
The following are the complete list of the connected graphs $G$
with seven edges such that $Z(G)\leq 13$:
 \begin{align*}
  Z(C(8))&=8,&
  Z(C(2,6))&= 13,&
  Z(Q_3(1,1,5))&= 12.
 \end{align*}
The Hosoya index of a graph with at least eight edges except $C(k)$
is greater than $13$.
Hence we have the following.
\begin{lemma}
\label{lemma:table}
Let $\ZZZ_n$ be the set of connected graphs  $G$ such that $Z(G)=n$.
We have
\begin{align*}
 \ZZZ_1 &=\Set{C(1)},\\
 \ZZZ_2 &=\Set{C(2)},\\
 \ZZZ_3 &=\Set{C(3)},\\
 \ZZZ_4&=\Set{C(4),\ Q_3(1,1,1)},\\
 \ZZZ_5&=\Set{C(2,2),\ C(5)},\\
 \ZZZ_6&=\Set{C(6),\ Q_3(1,1,2)},\\
 \ZZZ_7&=\Set{C(7),\ C(2,2,1),\ Q_4(1,1,1,1)},\\
 \ZZZ_8&=\Set{C(8),\ C(2,1,2),\ Q_3(1,1,3),\ \Theta_{3}(1,1,1,1)},\\
 \ZZZ_9&=\Set{C(9),\ C(2,4),\ Q_3(1,2,2)},\\
 \ZZZ_{10}&=\Set{
\begin{array}{c}
C(10),\ C(3,3),\ Q_3(1,1,1;2),\\
 Q_3(1,1,4),\ Q_4(1,1,1,2),\ K_{4}
\end{array}
},\\
 \ZZZ_{11}&=\Set{
\begin{array}{c}
C(11),\ C(2,1,3),\ C(2,5),\\
 Q_5(1,1,1,1,1),\ \Theta_{3}(1,1,2,1)
\end{array}
},\\
 \ZZZ_{12}&=\Set{
\begin{array}{c}
C(12),\ C(2,2,2),\ Q_3(1,2,3),\ Q_3(1,1,5),\\
 Q_3(1,1,\dot1)\vee Q_3(\dot1,1,1),\ \Theta_{3}(1,2,1,1)
\end{array}
},\\
 \ZZZ_{13}&=\Set{
\begin{array}{c}
C(13) ,\ C(2,6) ,\ C(3,4) ,\ C(2,1,1,2) ,\\
 \Theta_{3}(1,1,1,1,1) ,
\ Q_4(1,1,1,3) ,\ K_{2,3} 
\end{array}
}.
\end{align*} 
\end{lemma}

\subsection{Lemmas for the case where $(m,n)\in \PPP_1$}
Here we show lemmas to prove Theorem \ref{thm:main:1}.
First we show the following lemma to determine the shape of $B^{(1)}$.
  \begin{lemma}
   \label{lemma:B1}
   Let $A^{(1)}$, $B^{(1)}$, $G_{2,1}$ and $G_{4,1}$ be connected graphs
   satisfying
   \begin{align*}
    G_{2,1}\vee A^{(1)} \vee G_{2,1}&=C(3)\\
    Z(G_{2,1}\vee B^{(1)} \vee G_{2,1})&=4,\\
    Z(G_{4,1}\vee A^{(1)}\vee G_{4,1})&= 15.
   \end{align*}
   If $G_{2,1}\vee B^{(1)} \vee G_{2,1}\neq C(4)$,
   then     
   $Z(G_{4,1}\vee B^{(1)}\vee G_{4,1})$ does not equal $8$.
  \end{lemma}
  \begin{proof}
   Since $Z(G_{2,1}\vee B^{(1)} \vee G_{2,1})=4$,
   the graph $G_{2,1}\vee B^{(1)} \vee G_{2,1}$ 
   is either $Q_3(1,1,1)$ or $C(4)$.
   In this case where $G_{2,1}\vee B^{(1)} \vee G_{2,1}=Q_3(1,1,1)$, 
   the graphs
   $G_{2,1}$ and $B^{(1)}$ are $C(\dot1)$ and $Q_3(1,1,1)$, respectively.
   Since the graph $G_{2,1}\vee A^{(1)} \vee G_{2,1}$ is the caterpillar $C(3)$,
   the graph  $A^{(1)}$ is the caterpillar $C(3)$.
   Since    
   $Z(C(\dot2) \vee C(3) \vee C(\dot2)) \leq 8<15$
   for any base points, 
   the graph $G_{4,1}$ is not $C(\dot2)$.
   Since $G_{4,1}\neq C(\dot2)$,
   it follows that
   $Z(G_{4,1}\vee Q_3(1,1,1)\vee G_{4,1})\neq 8$.
  \end{proof}

  Next we show three lemmas to determine $C^{(1)}$.
  \begin{lemma}
   \label{lemma:C1a}
   For 
   $C^{(1)}=C(5)$ and $G_{4,1}=C(\dot3)$, 
   the Hosoya index
   $Z(G_{4,1}\vee C^{(1)} \vee G_{4,1})$ does not equal $17$ for any
   base points.
  \end{lemma}
  \begin{proof}
   It follows from direct calculation that
   \begin{align*}
    Z(G_{4,1}\vee C^{(1)}\vee G_{4,1})&=
    \begin{cases}
     Z(C(9))=9 &(C^{(1)}=C(\ddot5))\\
     Z(C(3,6))=19&(C^{(1)}=C(\dot1,\dot4))\\
     Z(C(3,3,3))=33&(C^{(1)}=C(\dot1,3,\dot1))\\
     Z(C(5,4))=21&(C^{(1)}=C(\ddot1,4)).
    \end{cases}
   \end{align*}  
  \end{proof}

  \begin{lemma}
   \label{lemma:C1b}
   Let $C^{(1)}=C(2,2)$ and $G_{4,1}=C(\dot3)$.
   If  $C^{(1)}\neq C(\dot2,\dot2)$ and $C^{(1)}\neq C(\ddot1,1,2)$,
   then 
   the Hosoya index $Z(G_{4,1}\vee C^{(1)} \vee G_{4,1})$
   does not equal $17$.
  \end{lemma}
  \begin{proof}
   It follows from direct calculation that
   \begin{align*}
    Z(G_{4,1}\vee C^{(1)} \vee G_{4,1})=
    \begin{cases}
     Z(C(3,3,2))=23&(C^{(1)}=C(\dot1,\dot1,2))\\
     Z(C(3,1,4))=19&(C^{(1)}=C(\dot1,1,\dot2))\\
     Z(C(3,1,1,3))=25&(C^{(1)}=C(\dot1,1,1,\dot1))\\
     Z(C(6,2))=13&(C^{(1)}=C(\ddot2,2)).
    \end{cases}
   \end{align*}
   \end{proof}
   
   \begin{lemma}
    \label{lemma:C1c}
    If $B^{(1)}=C(\ddot4)$ and $C^{(1)}=C(\ddot1,1,2)$,
    then there does not exist a connected graph $G_{5,2}$ such that
    $Z(G_{5,2}\vee B^{(1)}\vee G_{5,2})=20$ and
    $Z(G_{5,2}\vee C^{(1)}\vee G_{5,2})=29$.
   \end{lemma}
   \begin{proof}
    Let $G$ be a graph with a base point $v$, 
    and $G_v$ the restriction of $G$
    to vertices other than the vertex $v$.
    Since $B^{(1)}=C(\ddot4)$, we have
    \begin{align*}
     Z(G\vee B^{(1)}\vee G)=3Z(G_v)^2+Z(G\vee C(\ddot1) \vee G).
    \end{align*}
    Since $C^{(1)}=C(\dot1,1,2)$, we have
    \begin{align*}
     Z(G\vee C^{(1)}\vee G)=2Z(G_v)^2+3Z(G\vee C(\ddot1) \vee G).
    \end{align*}
     The following system of equation, however, has no integer solution:
    \begin{align*}
     \begin{pmatrix}
      3&1\\
      2&3
     \end{pmatrix}
     \begin{pmatrix}
      x_1\\
      x_2
     \end{pmatrix}
     =\begin{pmatrix}
       20\\
       29
      \end{pmatrix}.
    \end{align*}
\end{proof}

  \subsection{Lemmas for the case where $(m,n)\in \PPP_2$}
  Here we show lemmas to prove Theorem \ref{thm:main:2}.
  First we show the following lemma to determine the shape of $A^{(2)}$.
  \begin{lemma}
   \label{lemma:A2}
   Let $A^{(2)}$, $G_{3,2}$ and $G_{4,3}$ be connected graphs.
   Assume that $Z(G_{4,3}\vee A^{(2)}\vee G_{4,3})= 7$.
   If $G_{4,3}\vee A^{(2)}\vee G_{4,3}\neq C(7)$,
   then $Z(G_{3,2}\vee A^{(2)} \vee G_{3,2})$ does not equal $5$.   
  \end{lemma}
  \begin{proof}
   Since $Z(G_{4,3}\vee A^{(2)}\vee G_{4,3})=7$,
   the graph $G_{4,3}\vee A^{(2)}\vee G_{4,3}$ is  $C(2,3)$, $C(7)$ or $Q_4(1,1,1,1)$.
   
   Consider the case where $G_{4,3}\vee A^{(2)}\vee G_{4,3}=Q_4(1,1,1,1)$.
   In this case,  
   the graphs $A^{(2)}$ and $G_{4,3}$ are $Q_4(1,1,1,1)$ and
   $C(\dot1)$, respectively.
   Hence
   we have 
   \begin{align*}
    Z(G_{3,2}\vee A^{(2)} \vee G_{3,2}) \geq Z(A^{(2)}) =7.
   \end{align*}

   Consider the case where $G_{4,3}\vee A^{(2)}\vee G_{4,3}=C(2,3)$.
   In this case, candidates of $(A^{(2)},G_{4,3})$ are the following:
   \begin{align*}
    &(C(2,3),C(\dot1)),&
    &(C(2,\ddot 1),C(\dot2)),&
   &(C(\dot 1,\dot 2),C(\dot2)).&
   \end{align*}
   Since the difference of the numbers of edges of the graphs
   $G_{4,3}\vee A^{(2)}\vee G_{4,3}$ and
   $G_{3,2}\vee A^{(2)}\vee G_{3,2}$
   is even,
   the graph $G_{3,2}\vee A^{(2)}\vee G_{3,2}$ has even edges.
   Only $C(5)$ is a graph $G$ with even edges such that $Z(G)=5$.
   However,
   for any candidates of $A^{(2)}$,
   there does not exists $G_{3,2}$
   such that $G_{3,2}\vee A^{(2)}\vee G_{3,2}=C(5)$.
  \end{proof}

  Next we show three lemmas to determine the shape of $B^{(2)}$.
  \begin{lemma}
   \label{lemma:caseA}
   Let $A^{(2)}=C(\ddot5)$ and $G_{4,3}=C(\dot2)$.
   The Hosoya index
   $Z(G_{7,4}\vee B^{(2)} \vee G_{7,4})$
   does not equal
   $56$
   for connected graphs $B^{(2)}$ and $G_{7,4}$
   such that
   $Z(B^{(2)})=12$,
   $Z(G_{7,4}\vee A^{(2)} \vee G_{7,4})=33$, and 
   $Z(G_{4,3}\vee B^{(2)}\vee G_{4,3})= 24$.
  \end{lemma}
  
  \begin{proof}
   First consider the case where $G_{7,4}= C(\dot{15})$.
   In this case,
   we have $G_{7,4}\vee A^{(2)} \vee G_{7,4}=C(33)$,
   which implies $Z(G_{7,4}\vee A^{(2)} \vee G_{7,4})=33$.
   If we choose two different base points of $B^{(2)}$ 
   for the one-point union $G_{7,4}\vee B^{(2)} \vee G_{7,4}$,
   then the graph $G_{7,4}\vee B^{(2)} \vee G_{7,4}$ contains
   $C(15,1,\ldots,1,15)$ as a subgraph.
   Since 
   \begin{align*}
    Z(C(15,1,\ldots,1,15))\geq Z(C(15,15))=15\cdot 15+1>56,  
   \end{align*}
   we obtain 
   \begin{align*}
    Z(G_{7,4}\vee B^{(2)} \vee G_{7,4})>56.
   \end{align*}
   If we choose the same base points $v$ of $B^{(2)}$ 
   for left and right one-point unions of $G_{7,4}\vee B^{(2)} \vee G_{7,4}$,
   then
   \begin{align*}
    Z(G_{7,4}\vee B^{(2)} \vee G_{7,4})=Z( B^{(2)} \vee C(\dot{29})). 
   \end{align*}
   Let $B_v$ be 
   the restriction of $B^{(2)}$
    to  vertices other than the vertex $v$.
   Since 
    \begin{align*}
     Z( B^{(2)} \vee C(\dot{29}))=28Z(B_v)+Z(B^{(2)})=28Z(B_v)+12,
    \end{align*}
   we obtain 
   \begin{align*}
       Z(G_{7,4}\vee B^{(2)} \vee G_{7,4})\neq 56.
   \end{align*}

   Next we consider the case where
   $G_{7,4}\neq C(\dot{15})$.
   In this  case,
   there exits a vertex $u$ of the graph $G_{7,4}$ such that
   the distance between $u$ and the base point is greater than or equal
   to $2$.
   In other words,
   the graph $G_{7,4}$ contains $C(\dot1,2)$ as a subgraph with a base point.
   Hence
   \begin{align*}
    Z(G_{7,4}\vee B^{(2)} \vee G_{7,4})\geq Z(C(\dot1,2) \vee B^{(2)}\vee C(\dot1,2)). 
   \end{align*}

   Let $B_{u}$ (\textit{resp.} $B_{v}$, $B_{u,v}$)  be 
   the restriction of $B^{(2)}$
   to vertices other than 
   the left (\textit{resp.} right, both) 
   base point of $B^{(2)}$ for the one-point sum 
   $G_{7,4}\vee B^{(2)} \vee G_{7,4}$.
   If $u\neq v$, then we have 
   \begin{align*}
    Z(C(\dot1,2) \vee B^{(2)}\vee C(\dot1,2))= 4Z(B^{(2)})+2Z(B_u)+2Z(B_v)+B_{u,v}. 
   \end{align*}
   Since $Z(B^{(2)})=12$, we have
   \begin{align*}
    &Z(G_{7,4}\vee B^{(2)} \vee G_{7,4})\\
    &\geq Z(C(\dot1,2) \vee B^{(2)}\vee C(\dot1,2))\\
    &= 4Z(B^{(2)})+2Z(B_u)+2Z(B_v)+B_{u,v}\\
    &\geq 48+2Z(B_v)+2Z(B_u)+1.
   \end{align*}
   If $Z(B_v),Z(B_u)\geq 2$, then 
   $Z(G_{7,4}\vee B^{(2)} \vee G_{7,4})>57$.
   If $Z(B_v)\leq 1$ and $Z(B^{(2)})=12$, then 
   the graph $B^{(2)}$ is $C(\dot1,\dot{11})$.
   However, in this case, $B_u=C(11)$ and $Z(B_u)=11\geq  2$.

   If $u=v$, then 
   the graph $C(\dot1,2) \vee B^{(2)}\vee C(\dot1,2)$
   is $B^{(2)}\vee C(2,\dot1,2)$.
   Since $Z(B^{(2)}\vee C(2,\dot1,2))=4Z(B_u)+4Z(B^{(2)})$,
   we have 
   \begin{align*}
    Z(G_{7,4}\vee B^{(2)} \vee G_{7,4})
    &\geq 4Z(B_u)+4Z(B^{(2)})\\
    &=4Z(B_u)+48.
   \end{align*}
   Hence,
   if $B^{(2)}\not\in\Set{ C(\ddot k), Q_3(1,1,\ddot k), C(\ddot k,2)}$,
   then 
   \begin{align*}
    Z(C(\dot1,1,1)\vee B^{(2)}\vee C(\dot1,1,1))>56.
   \end{align*}
   If the graph $B^{(2)}$ is 
   $C(\ddot k)$, $Q_3(1,1,\ddot k)$ or $C(\ddot k,2)$,
   then 
   the graph $B^{(2)}$ is 
   either $C(\ddot{12})$ or $Q_3(1,1,5)$
   since $Z(B^{(2)})=12$.
   However,
   since $Z(C(14))=14$ and $Z(Q_3(1,1,7))=16$,
   these two graphs do not satisfy the equation
   $Z(G_{4,3}\vee B^{(2)}\vee G_{4,3}) =24$.
  \end{proof}

  \begin{lemma}
   \label{lemma:caseC}
   Let $G_{3,2}=C(\dot3)$, $G_{4,3}=C(\dot4)$.
   For a connected graph $B^{(2)}$ satisfying
   $Z(G_{3,2}\vee B^{(2)} \vee G_{3,2})=12$,
   the Hosoya index $Z(G_{4,3}\vee B^{(2)}\vee G_{4,3})$
   does not equal
   $24$.
  \end{lemma}
  \begin{proof}  
   Since $G_{3,2}=C(\dot3)$ and $Z(G_{3,2}\vee B^{(2)} \vee G_{3,2})=12$,
   the graph $B^{(2)}$ is either $Q_3(1,1,\ddot1)$ or $C(\ddot{8})$.
   Hence 
   $Z(G_{4,3}\vee B^{(2)}\vee G_{4,3})$ 
   equals $Z(Q_3(1,1,7))=16$ or $Z(C(14))=14$.
  \end{proof}

  \begin{lemma}
   \label{lemma:caseB}
   Let
   $G_{3,2}=C(\dot2)$, and $G_{4,3} = C(\dot3)$.
   Let $B^{(2)}$ be a connected graph such that
   $Z(G_{3,2}\vee B^{(2)}\vee G_{3,2})=12$.
   If
   $B^{(2)}\neq C(\dot1,2,\dot1)$,
   then
   $Z(G_{4,3}\vee B^{(2)}\vee G_{4,3})$ does not equal $24$.
  \end{lemma}
  \begin{proof}
    First we consider the case where 
   the graph $G_{3,2}\vee B^{(2)} \vee G_{3,2}$ is 
   either $Q_3(1,1,5)$ or $C(12)$.
    In this case, 
   similarly to Lemma \ref{lemma:caseC},
   it follows that
    $Z(G_{4,3}\vee B^{(2)}\vee G_{4,3})$ is $Z(Q_3(1,1,7))=16$ or $Z(C(13))=14$,
    which are not equal to $24$.

   Next we consider the case  where 
   the graph $G_{3,2}\vee B^{(2)} \vee G_{3,2}$ is neither $Q_3(1,1,5)$
   nor $C(12)$.
   Since 
   \begin{align*}
   Z(G_{3,2}\vee B^{(2)} \vee G_{3,2})=Z(C(\dot2)\vee B^{(2)}\vee C(\dot2))=12, 
   \end{align*}   
   candidates of the graph $G_{3,2}\vee B^{(2)} \vee G_{3,2}$
   are $Q_3(3,1,2)$ and $C(2,2,2)$.
    Hence candidates of $B^{(2)}$  are the following:
   \begin{align*}
    &Q_3(\ddot 1,1,2),&
    &Q_3(\dot 2,1,\dot 1),&
    &C(\dot 1,\dot 1,2),&
    &C(\dot 1,2,\dot 1).
   \end{align*}
   Since $G_{4,3}=C(\dot3)$, 
   we have
   \begin{align*}
    Z(G_{4,3} \vee B^{(2)} \vee G_{4,3})=
    \begin{cases}
     Z(Q_3(5,1,2))=18 &(B^{(2)}=Q_3(\ddot1,1,2))\\
     Z(Q_3(4,1,3))=20 &(B^{(2)}=Q_3(\dot2,1,\dot1))\\
     Z(C(3,3,2))=23 &(B^{(2)}=C(\dot1,\dot1,2)),
    \end{cases}
   \end{align*}
   which implies $Z(G_{4,3} \vee B^{(2)} \vee G_{4,3})\neq 24$.
  \end{proof}

  Next we show a lemma to determine $G_{7,4}$.
  \begin{lemma}
   \label{lemma:G74}
   Let $G_{7,4}$ and $B^{(2)}$ be connected graphs.
   Assume that  $B^{(2)}=C(\dot1,2,\dot1)$.
   If $G_{7,4}\neq C(\dot1,3)$,
   then $Z(G_{7,4}\vee B^{(2)}\vee G_{7,4})$ does not equal $56$.
  \end{lemma}
  \begin{proof}
   First consider the case where the graph $G_{7,4}$ 
   contains $C(\dot 3)$ as a subgraph  with a base point.
   If the graph $G_{7,4}$ contains 
   $C(\dot2,2)$ as a subgraph with a base points,
   then 
   \begin{align*}
    Z(G_{7,4}\vee B^{(2)}\vee G_{7,4})&\geq Z(C(2,2,2,2,2))\\
    &\geq
    Z(C(2,2))\cdot Z(C(2,2,2))=5\cdot12>56.
   \end{align*}
   If the graph $G_{7,4}$ 
   contains $Q_3(1,1,\dot 2)$ as a subgraph with a base points,
   then 
   \begin{align*}
    Z(G_{7,4}\vee B^{(2)}\vee G_{7,4})
&\geq  Z(Q_3(1,1,2))\cdot Z(Q_3(1,1,2;1,1))\\
&=6\cdot 14>56.
   \end{align*}
   If $G_{7,4}=Q_3(1,1,\dot 1)$, then 
   \begin{align*}
    Z(G_{7,4}\vee B^{(2)}\vee G_{7,4})=48<56.
   \end{align*}
   Therefore, if the graph $G_{7,4}$ contains $C(\dot 3)$ as a subgraph  with a
   base point,
   then $Z(G_{7,4}\vee B^{(2)}\vee G_{7,4})\neq 56$.

   Next consider the case where
   the graph $G_{7,4}$ 
   does not contain $C(\dot 3)$ as a subgraph  with a base point.
   If $G_{7,4}=C(\dot1,1,2)$, then
   \begin{align*}
    Z(G_{7,4}\vee B^{(2)}\vee G_{7,4})=Z(C(2,1,1,2,1,1,2))=5\cdot 13+3\cdot 5>56.
   \end{align*}
   If $G=C(\dot1,3)$, then
   $Z(G_{7,4}\vee B^{(2)}\vee G_{7,4})=Z(C(3,1,2,1,3))=56$.
   Hence, if 
   the graph 
   $G_{7,4}$ contains $C(\dot1,1,2)$ as a proper subgraph  with a
   base point, then 
   $Z(G_{7,4}\vee B^{(2)}\vee G_{7,4})>56$.
  \end{proof}

  Finally we show the following lemma to determine $C^{(2)}$.
  \begin{lemma}
   \label{lemma:C2}
   Let $G_{3,2}=C(\dot2)$,
   $G_{4,3}=C(\dot3)$,
   $G_{7,4}=C(\dot1,3)$.
   Let $C^{(2)}$ be a connected graphs satisfying
   $Z(G_{3,2}\vee C^{(2)}\vee G_{3,2})=13$ and 
   $Z(G_{7,4}\vee C^{(2)} \vee G_{7,4})=65$.
   If $C^{(2)}\neq C(\dot1,1,1,\dot1)$,
   then
   $Z(G_{4,3}\vee C^{(2)} \vee G_{4,3})$ does not equal $25$,
  \end{lemma}
  \begin{proof}
   It follows that 
   \begin{align*}
   &Z(G_{7,4}\vee C^{(2)}\vee G_{7,4})\\
    =& Z(C(\dot 1,3) \vee C^{(2)} \vee C(\dot 1,3))\\
    =&  Z(C(\dot 1,3) \vee C^{(2)} \vee C(\dot 1,1))    
      +2 Z(C(\dot 1,3) \vee C^{(2)})\\    
    =&  Z(C(\dot 1,1) \vee C^{(2)} \vee C(\dot 1,1))    
    +2 Z( C^{(2)}\vee C(\dot 1,1))\\
    &+2 Z(C(\dot 1,1) \vee C^{(2)}) 
     +4 Z(C^{(2)}) \\
  \geq& Z(C(\dot 1,1) \vee C^{(2)} \vee C(\dot 1,1))+8Z(C^{(2)}) .
   \end{align*}
   Since $G_{3,2}=C(\dot 1,1)$ and $Z(G_{3,2}\vee C^{(2)}\vee G_{3,2})=13$,
   we have
   \begin{align*}
    65=Z(G_{7,4}\vee C^{(2)}\vee G_{7,4})
    \geq 13+8Z(C^{(2)}).
   \end{align*}
   Hence $Z(C^{(2)})\leq 6$.

   Consider the case where $C^{(2)}=C(6)$.
   Since the graph $G_{3,2}$ is $C(\dot2)$,
   we have
   \begin{align*}
     Z(G_{3,2}\vee C^{(2)} \vee G_{3,2})
    =
    \begin{cases}
     Z(C(8)) =8  &(C^{(2)}=C(\ddot6)) \\
     Z(C(6,2)) =13  &(C^{(2)}= C(\dot5,\dot1)) \\
     Z(C(5,3)) =16  &(C^{(2)}=C(5,\ddot 1)) \\
     Z(C(2,4,2)) =20  &(C^{(2)}=C(\dot1,4,\dot1)) .
     \end{cases}
   \end{align*}
   Hence the graph $C^{(2)}$ is $C(\dot5,\dot1)$.
   On the other hand,
   since the graph $G_{4,3}$ is $C(\dot3)$,
   we have $Z(G_{4,3}\vee C^{(2)} \vee G_{4,3})=Z(C(3,7))=22$.

   Consider the case where $C^{(2)}\neq C(6)$.
   Since $Z(C^{(2)})\leq 6$, the number of edges in the graph $C^{(2)}$
   is less than or equal to $4$.
   Since 
   $Z(G_{3,2}\vee C^{(2)} \vee G_{3,2})=13$,
   candidates of the graph
   $G_{3,2}\vee C^{(2)} \vee G_{3,2}=C(\dot2) \vee C^{(2)} \vee C(\dot2)$ are 
   the following:
   \begin{align*}
    &C(2,1,1,2), &
    &C(3,4), &
    &Q_4(1,1,1,3) .
   \end{align*}
   Since $Z(Q_4(1,1,1,1))=7>4$,
   the graph $C^{(2)}$ is not $Q_4(1,1,1,1)$.
   If the graph $G_{3,2}\vee C^{(2)} \vee G_{3,2}$ 
   is the caterpillar $C(2,1,1,2)$,
   then the graph $C^{(2)}$ is 
   $C(\dot 1,1,1,\dot 1)$.
   Hence we 
   consider only the case where $G_{3,2}\vee C^{(2)} \vee G_{3,2}=C(3,4)$.
   In this case, candidates of  the graph $C^{(2)}$ are the following:
   \begin{align*}
    &C(\ddot 1,4), &
    &C(1,\dot 1,\dot 2,1), &
    &C( 3,\ddot 2) .
   \end{align*}
   Hence we have
   \begin{align*}
    Z(G_{4,3}\vee C^{(2)} \vee G_{4,3})
    =
    \begin{cases}
     Z(5,3,1)=21 &(C^{(2)}=C(\ddot 1,4))\\
     Z(4,4,1)=21 &(C^{(2)}=C(\dot 2,\dot 3))\\
     Z(1,2,6)=19 &(C^{(2)}=C(3,\ddot 2)).
    \end{cases}
   \end{align*}
  \end{proof}

\bibliographystyle{amsplain-url}
\bibliography{by-mr}

\end{document}